\documentclass{amsproc}
\usepackage{mathrsfs,amsfonts,amsmath,amsthm,amssymb}
\usepackage{graphics}
\usepackage[all]{xy}
\xyoption{curve}
\xyoption{import}
\xyoption{arc}
\xyoption{ps}
\UsePSspecials{dvips}
\usepackage{amsmath}
\usepackage{graphicx}

\title{Higher localized analytic indices and strict deformation quantization}
\author{Paulo Carrillo Rouse}

\begin{document}


{\theoremstyle{definition}\newtheorem{definition}{Definition}[section]
\newtheorem{notation}[definition]{Notation}
\newtheorem{remnot}[definition]{Remarks and notation}
\newtheorem{terminology}[definition]{Terminology}
\newtheorem{remark}[definition]{Remark}
\newtheorem{remarks}[definition]{Remarks}
\newtheorem{example}[definition]{Example}
\newtheorem{examples}[definition]{Examples}}
\newtheorem{proposition}[definition]{Proposition}
\newtheorem{lemma}[definition]{Lemma}
\newtheorem{theorem}[definition]{Theorem}
\newtheorem{corollary}[definition]{Corollary}
\newtheorem*{teorema}{Theorem}
\newtheorem*{teorema1}{Theorem 1}

\newcommand{\ci}{C^{\infty}}
\newcommand{\A}{\mathscr{A}}
\newcommand{\Cat}{\mathscr{C}}
\newcommand{\Dnc}{\mathscr{D}}
\newcommand{\E}{\mathscr{E}}
\newcommand{\F}{\mathscr{F}}
\newcommand{\gr}{\mathscr{G}}
\newcommand{\go}{\mathscr{G} ^{(0)}}
\newcommand{\hr}{\mathscr{H}}
\newcommand{\ho}{\mathscr{H} ^{(0)}}
\newcommand{\gd}{\mathscr{G}^{\mathbb{R}^2}}
\newcommand{\gt}{\mathscr{G} ^{T}}
\newcommand{\I}{\mathscr{I}}
\newcommand{\Nb}{\mathscr{N}}
\newcommand{\Kom}{\mathscr{K}}
\newcommand{\ops}{\mathscr{O}}
\newcommand{\Pb}{\mathscr{P}}
\newcommand{\sw}{\mathscr{S}}
\newcommand{\Uo}{\mathscr{U}}
\newcommand{\Vo}{\mathscr{V}}
\newcommand{\Rr}{\mathbb{R}}
\newcommand{\Nat}{\mathbb{N}}
\newcommand{\src}{\mathscr{S}_{c}}
\newcommand{\cc}{C_{c}^{\infty}}
\newcommand{\cg}{C_{c}^{\infty}(\gr)}
\newcommand{\cgo}{C_{c}^{\infty}(\go)}
\newcommand{\ct}{C_{c}^{\infty}(\gr^T)}
\newcommand{\ckt}{C_{c}^{k}(\gr \times [0,1])}
\newcommand{\ck}{C_{c}^{k}(\gr)}
\newcommand{\ca}{C_{c}^{\infty}(A\gr)}
\newcommand{\Un}{{U}^{(n)}}
\newcommand{\Du}{D_{\mathscr{U}}}






\maketitle

\begin{abstract}
This paper is concerned with the localization of higher analytic indices for Lie groupoids. Let $\gr$ be a Lie groupoid with Lie algebroid $A\gr$. 
Let $\tau$ be a (periodic) cyclic cocycle over the convolution algebra $\cg$. We say that $\tau$ can be localized if there is a correspondence 
\begin{equation}\nonumber
K^0(A^*\gr)\stackrel{Ind_{\tau}}{\longrightarrow}\mathbb{C}
\end{equation}
satisfying $Ind_{\tau}(a)=\langle ind\, D_a,\tau \rangle$ (Connes pairing). In this case, we call $Ind_{\tau}$ the higher localized index associated to $\tau$. 
In \cite{Ca4} we use the algebra of functions over the tangent groupoid introduced in \cite{Ca2}, which is in fact a strict deformation quantization of the Schwartz algebra $\sw(A\gr)$, to prove the following results:
\begin{itemize}
\item Every bounded continuous cyclic cocycle can be localized.
\item If $\gr$ is {\'e}tale, every cyclic cocycle can be localized.
\end{itemize}
We will recall this results with the difference that in this paper, a formula for higher localized indices will be given in terms of an asymptotic limit of a pairing at the level of the deformation algebra mentioned above. We will discuss how the higher index formulas of Connes-Moscovici, Gorokhovsky-Lott fit in this unifying setting.
\end{abstract}

\tableofcontents

\section{Introduction}

This paper is concerned with the localization of higher analytic indices for Lie groupoids. In \cite{CMnov}, Connes and Moscovici defined, for any smooth manifold $M$ and every Alexander-Spanier class $[\bar{\varphi}]\in \bar{H}_{c}^{ev}(M)$, a localized index morphism
\begin{equation}\label{locindCM}
Ind_{\varphi}:K^0_c(T^*M)\longrightarrow \mathbb{C}.
\end{equation}
which has as a particular case the analytic index morphism of Atiyah-Singer for $[1]\in \bar{H}^0(M)$.

Indeed, given an Alexander-Spanier cocycle $\varphi$ on $M$, Connes-Moscovici construct a cyclic cocycle 
$\tau(\varphi)$ over the algebra of smoothing operators, $\Psi^{-\infty}(M)$ (lemma 2.1, ref.cit.). Now, if $D$ is an elliptic pseudodifferential operator over $M$, it defines an index class $ind\, D\in K_0(\Psi^{-\infty}(M))\approx \mathbb{Z}$. Then they showed (theorem 2.4, ref.cit.) that the pairing 
\begin{equation}
\langle ind\, D, \tau(\varphi)\rangle
\end{equation}
only depends on the principal symbol class $[\sigma_D]\in K^0(T^*M)$ and on the class of $\varphi$, and this defines the localized index morphism (\ref{locindCM}). Connes-Moscovici go further to prove a localized index formula generalizing the Atiyah-Singer theorem. They used this formula to prove the so called Higher index theorem for coverings which served for proving the Novikov conjecture for Hyperbolic groups.

We discuss now the Lie groupoid case. This concept is central in non commutative geometry. Groupoids
generalize the concepts of spaces, groups and 
equivalence relations. In the late 70's, mainly with the work
of Alain Connes, it became clear that groupoids appeared naturally as
substitutes of singular spaces \cite{Coinc,Mac,Ren, Pat}. Many people have contributed to realizing this idea. We can
find for instance a groupoid-like treatment in Dixmier's works on
transformation groups, \cite{Dixmier77}, or in Brown-Green-Rieffel's work on orbit
classification of relations, \cite{BGR}. 
In foliation theory, several models for
the leaf space of a foliation were realized using groupoids, mainly
by people like Haefliger (\cite{Haef}) and Wilkelnkemper (\cite{Win}), to mention
some of them. There is also the case of Orbifolds, these can
be seen indeed as {\'e}tale groupoids, (see for example Moerdijk's paper 
\cite{MoerOrb}). There are also some particular groupoid models for manifolds with
corners and conic manifolds worked by people like Monthubert \cite{Mont},
Debord-Lescure-Nistor (\cite{DLN}) 
and Aastrup-Melo-Monthubert-Schrohe (\cite{AMMS}) for example. Furthermore, Connes showed that many
groupoids and algebras associated to them appeared as `non commutative
analogues` of smooth manifolds to which many tools of geometry such as
K-theory and Characteristic classes could be applied \cite{Coinc,Concg}. 
Lie groupoids became a very natural place where to perform pseudodifferential calculus and index theory, \cite{Coinc,MP,NWX}. 

The study of the indices in the groupoid case is, as we will see, more delicate than the classical case. There are new phenomena appearing.
If $\gr$ is a Lie groupoid, a $\gr$-pseudodifferential operator is a differentiable family (see \cite{MP,NWX}) of operators. Let $P$ be such an operator, the index of $P$, $ind\, P$, is an element of $K_0(\cg)$. We have also a Connes-Chern pairing 
$$K_0(\cg)\times HC^{\tiny{even}}(\cg)\stackrel{\langle\_,\_\rangle}{\longrightarrow}\mathbb{C}.$$
We would like to compute the pairings of the form
\begin{equation}\label{pairind}
\langle ind\,D,\tau \rangle
\end{equation}
for $D$ a $\gr$-pseudodifferential elliptic operator. For instance, the Connes-Mosocovici Higher index theorem gives a formula for the above pairing when the groupoid $\gr$ is the groupoid associated to a $\Gamma$-covering and for cyclic group cocycles.

Now, the first step in order to give a formula for the pairing (\ref{pairind}) above is to localize the pairing, that is, to show that it only depends on the principal symbol class in $K^0(A^*\gr)$ (this would be the analog of theorem 2.4, \cite{CMnov}).

Let $\tau$ be a (periodic) cyclic cocycle over $\cg$. We say that $\tau$ can be localized if 
the correspondence
\begin{equation}
\xymatrix{
Ell(\gr)\ar[r]^-{ind}&K_0(\cg)\ar[r]^-{\langle\, \_, \tau \rangle}&\mathbb{C}
}
\end{equation}
factors through the principal symbol class morphism, where $Ell(\gr)$ is the set of $\gr$-pseudodifferential elliptic operators. In other words, if there is a unique morphism $K^0(A^*\gr)\stackrel{Ind_{\tau}}{\longrightarrow}\mathbb{C}$ which fits in the following commutative diagram
\begin{equation}
\xymatrix{
Ell(\gr)\ar[r]^-{ind} \ar[d]_{[psymb]}&K_0(\cg)\ar[r]^-{\langle\, \_, \tau \rangle}&\mathbb{C}\\
K^0(A^*\gr)\ar[rru]_-{Ind_{\tau}}&&
}
\end{equation}
{\it i.e.}, satisfying $Ind_{\tau}(a)=\langle ind\, D_a,\tau \rangle$, and hence completely characterized by this property. 
In this case, we call $Ind_{\tau}$ the higher localized index associated to $\tau$.

In this paper, we prove a localization result using an appropriate strict deformation quantization algebra. For stating the main theorem we need to introduce some terms.

Let $\gr$ be a Lie groupoid. It is known that the topological $K$-theory group $K^0(A^*\gr)$, encodes the classes of principal symbols of all $\gr$-pseudodifferential elliptic operators, \cite{AS}. On other hand the K-theory of the Schwartz algebra of the Lie algebroid satisfies 
$K_0(\sw (A\gr))\approx K^0(A^*\gr)$ (see for instance \cite{Ca4} Proposition 4.5).

In \cite{Ca2}, we constructed a strict deformation quantization of the algebra $\sw (A\gr)$. This algebra is based on the notion of the tangent groupoid which is a deformation groupoid associated to any Lie groupoid: Indeed, associated to a Lie groupoid $\gr \rightrightarrows \go$, there is a Lie groupoid
$$\gr^T:=A\gr \times \{ 0\} \bigsqcup \gr \times (0,1]
\rightrightarrows \go \times [0,1],$$
compatible with $A\gr$ and $\gr$, called the tangent groupoid of $\gr$. We can now recall the main theorem in ref.cit.

\begin{teorema}
There exists an intermediate algebra $\src (\gr^T)$ consisting of smooth functions over the tangent groupoid
$$\ct \subset \src (\gr^T) \subset C_r^*(\gr^T),$$ such that it is a field of algebras over $[0,1]$, whose fibers are
$$\sw (A\gr) \text{ at } t=0, \text{ and }$$
$$\cg \text{ for } t\neq 0.$$  
\end{teorema} 

Let $\tau$ be a $(q+1)-$multilinear functional over $\cg$. For each $t\neq 0$, we let 
$\tau_t$ be the $(q+1)$-multilinear functional over $\src (\gt)$ defined by
\begin{equation}
\tau_t(f^0,...,f^q):=\tau(f^0_t,...,f^q_t).
\end{equation}
It is immediate that if $\tau$ is a (periodic) cyclic cocycle over $\cg$, then $\tau_t$ is a (periodic) cyclic cocycle over $\src(\gt)$ for each $t\neq 0$.

The main result of this paper is the following:

\begin{teorema1}
Every bounded cyclic cocycle can be localized. Moreover, in this case, the following formula for the higher localized index holds:
\begin{equation}\label{tauindlocintro}
Ind_{\tau}(a)=lim_{t\rightarrow 0}\langle \widetilde{a},\tau_t \rangle,
\end{equation}
where $\widetilde{a}\in K_0(\src(\gt))$ is such that  $e_0(\widetilde{a})=a\in K^0(A^*\gr)$.
In fact the pairing above is constant for $t\neq 0$.
\end{teorema1}

Where, a multilinear map $\tau:\underbrace{\cg \times \cdots \times \cg}_{q+1-times}
\rightarrow \mathbb{C}$ is bounded if it extends to a continuous multilinear map 
$\underbrace{\ck \times \cdots \times \ck}_{q+1-times}\stackrel{\tau_k}{\longrightarrow}\mathbb{C}$, for some $k\in \mathbb{N}$. The restriction of taking bounded continuous cyclic cocycles in the last theorem is not at all restrictive. In fact, all the geometrical cocycles are of this kind (Group cocycles, The transverse fundamental class, Godbillon-Vey and all the Gelfand-Fuchs cocycles for instance).

Moreover, for the case of {\'e}tale groupoids, the  explicit calculations of the Periodic cohomologies spaces developed in \cite{BN,Cra} allow us to conclude that the formula (\ref{tauindlocintro}) above holds for every cyclic cocycle (Corollary \ref{corindloc}).

At the end of this work we will discuss how the higher index formulas of Connes-Moscovici, Gorokhovsky-Lott (\cite{CMnov,GorLotteg}) fit in this unifying setting.

\vspace{2mm}
\noindent
{\bf Acknowledgments} 
I would like to thank Georges Skandalis for reading an earlier version of this paper and for the very useful comments and remarks he did to it. 

I would also like to thank the referee for his remarks to improve this work.

\section{Index theory for Lie groupoids}
\subsection{Lie groupoids}

Let us recall what a groupoid is:

\begin{definition}
A $\it{groupoid}$ consists of the following data:
two sets $\gr$ and $\go$, and maps
\begin{itemize}
\item[$\cdot$]$s,r:\gr \rightarrow \go$ 
called the source and target map respectively,
\item[$\cdot$]$m:\gr^{(2)}\rightarrow \gr$ called the product map 
(where $\gr^{(2)}=\{ (\gamma,\eta)\in \gr \times \gr : s(\gamma)=r(\eta)\}$),
\end{itemize}
such that there exist two maps, $u:\go \rightarrow \gr$ (the unit map) and 
$i:\gr \rightarrow \gr$ (the inverse map),
such that, if we denote $m(\gamma,\eta)=\gamma \cdot \eta$, $u(x)=x$ and 
$i(\gamma)=\gamma^{-1}$, we have 
\begin{itemize}
\item[1.]$r(\gamma \cdot \eta) =r(\gamma)$ and $s(\gamma \cdot \eta) =s(\eta)$.
\item[2.]$\gamma \cdot (\eta \cdot \delta)=(\gamma \cdot \eta )\cdot \delta$, 
$\forall \gamma,\eta,\delta \in \gr$ when this is possible.
\item[3.]$\gamma \cdot x = \gamma$ and $x\cdot \eta =\eta$, $\forall
  \gamma,\eta \in \gr$ with $s(\gamma)=x$ and $r(\eta)=x$.
\item[4.]$\gamma \cdot \gamma^{-1} =u(r(\gamma))$ and 
$\gamma^{-1} \cdot \gamma =u(s(\gamma))$, $\forall \gamma \in \gr$.
\end{itemize}
Generally, we denote a groupoid by $\gr \rightrightarrows \go $.
\end{definition}

Along this paper we will only deal with Lie groupoids, that is, 
a groupoid in which $\gr$ and $\go$ are smooth manifolds (possibly with boundary), and $s,r,m,u$ are smooth maps (with s and r submersions, see \cite{Mac,Pat}). For $A,B$ subsets of $\go$ we use the notation
$\gr_{A}^{B}$ for the subset $\{ \gamma \in \gr : s(\gamma) \in A,\, 
r(\gamma)\in B\}$.

Our first example of Lie groupoids will be the Lie groups, we will give other examples below.

\begin{example}[Lie Groups]
Let $G$ be a Lie group. Then $$G\rightrightarrows \{ e \}$$ is a Lie groupoid with product given by the group product, the unit is the unit element of the group and the inverse is the group inverse
\end{example}

Lie groupoids generalize Lie groups. Now, for Lie groupoids there is also a notion playing the role of the Lie algebra:

\begin{definition}[The Lie algebroid of a Lie groupoid]
Let $\gr \rightarrow \go$ be a Lie groupoid. The Lie algebroid of $\gr$ is the vector bundle  
$$A\gr \rightarrow \go$$
given by definition as the normal vector bundle associated to the inclusion $\go \subset \gr$ 
(we identify $\go$ with its image by $u$). 
\end{definition} 

For the case when a Lie groupoid is given by a Lie group as above $G\rightrightarrows \{e\}$, we recover
$AG=T_eG$. 
Now, in the Lie theory is very important that this vector space, $T_eG$, has a Lie algebra structure. In the setting of Lie groupoids the Lie algebroid $A\gr$ has a structure of Lie algebroid. We will not need this in this paper.

Let us put some classical examples of Lie groupoids.

\begin{example}[Manifolds]
Let $M$ be a $\ci$-manifold. We can consider the groupoid 
$$M\rightrightarrows M$$ where every morphism is the identity over $M$.
\end{example}

\begin{example}[Groupoid associated to a manifold]
Let $M$ be a $\ci$-manifold. We can consider the groupoid 
$$M\times M\rightrightarrows M$$ with $s(x,y)=y$, $r(x,y)=x$ and the product given by 
$(x,y)\circ (y,z)=(x,z)$. We denote this groupoid by $\gr_M$.
\end{example}

\begin{example}\label{grprsub}[Fiber product groupoid associated to a submersion]
This is a generalization of the example above. Let $N\stackrel{p}{\rightarrow} M$ be a submersion. We consider the fiber product $N\times_M N:=\{ (n,n')\in N\times N :p(n)=p(n') \}$,which is a manifold because $p$ is a submersion. We can then take the groupoid 
$$N\times_M N\rightrightarrows N$$ which is only a subgroupoid of
$N\times N$. 
\end{example}

\begin{example}[G-spaces]
Let $G$ be a Lie group acting by diffeomorphisms in a manifold $M$. The transformation groupoid associated to this action is
$$M \rtimes G\rightrightarrows M.$$ As a set $M\rtimes
G=M\times G $, and the maps are given by 
$s(x,g)=x\cdot g$, $r(x,g)=x$, the product given by 
$(x,g)\circ (x\cdot g,h)=(x,gh)$, the unit is $u(x)=(x,e)$ and with inverse $(x,g)^{-1}=(x\cdot g,g^{-1})$.
\end{example}

\begin{example}[Vector bundles]
Let $E\stackrel{p}{\rightarrow} X$ be a smooth vector bundle over a manifold $X$. We consider the groupoid 
$$E\rightrightarrows X$$ with $s(\xi)=p(\xi)$, $r(\xi)=p(\xi)$, the product uses the vector space structure and it is given by $\xi \circ \eta =\xi +\eta$, the unit is zero section and the inverse is the additive inverse at each fiber.
\end{example}

\begin{example}[Haefliger's groupoid]
Let $q$ be a positive integer. The Haefliger's groupoid $\Gamma^q$ has as space of objects $\Rr^q$. A morphism (or arrow) $x\mapsto y$ in $\Gamma^q$ is the germ of a (local) diffeomorphism $(\Rr^q,x)\rightarrow (\Rr^q,y)$. This Lie groupoid and its classifying space play a vey important role in the theory of foliations, \cite{Haef}.
\end{example}

\begin{example}[Orbifolds]
An Orbifold is an {\'e}tale groupoid for which $(s,r):\gr \rightarrow \go \times \go$ is a proper map.
See \cite{MoerOrb} for further details.
\end{example}

\begin{example}[Groupoid associated to a covering]
Let $\Gamma$ be a discret group acting freely and properly in $\widetilde{M}$ with compact quotient 
$\widetilde{M}/\Gamma:=M$. 
We denote by $\gr$ the quotient $\widetilde{M}\times \widetilde{M}$ 
by the diagonal action of $\Gamma$. We have a Lie groupoid 
$\gr \rightrightarrows \go=M$ with $s(\widetilde{x},\widetilde{y})=y$, 
$r(\widetilde{x},\widetilde{y})=x$ and product 
$(\widetilde{x},\widetilde{y})\circ(\widetilde{y},\widetilde{z})
=(\widetilde{x},\widetilde{z})$. 

A particular of this situation is when 
$\Gamma=\pi_1(M)$ and $\widetilde{M}$ is the universal covering. This groupoid played a main role in the Novikov's conjecture proof for hyperbolic groups given by Connes and Moscovici, \cite{CMnov}.
\end{example}

\begin{example}[Holonomy groupoid of a Foliation]
Let $M$ be a compact manifold of dimension $n$. Let $F$ be a subvector bundle of the tangent bundle $TM$.
We say that $F$ is integrable if  
$\ci(F):=\{ X\in \ci(M,TM): \forall x\in M, X_x\in F_x\}$ is a Lie subalgebra of $\ci(M,TM)$. This induces a partition in embedded submanifolds (the leaves of the foliation), given by the solution of integrating $F$. 

The holonomy groupoid of $(M,F)$ is a Lie groupoid  
$$\gr(M,F)\rightrightarrows M$$ with Lie algebroid $A\gr=F$ and minimal in the following sense: 
any Lie groupoid integrating the foliation \footnote{having $F$ as Lie algebroid} contains an open subgroupoid which maps onto the holonomy groupoid by a smooth morphism of Lie groupoids. 

The holonomy groupoid was constructed by Ehresmann \cite{Ehr} and Winkelnkemper \cite{Win} (see also 
\cite{Candel}, \cite{God}, \cite{Pat}).
\end{example}

\subsubsection{The convolution algebra of a Lie groupoid}

We recall how to define an algebra structure in $\cg$ using
smooth Haar systems.
 
\begin{definition}
A $\it{smooth\, Haar\, system}$ over a Lie groupoid is a family of
measures $\mu_x$ in $\gr_x$ for each $x\in \go$ such that,

\begin{itemize}
\item for $\eta \in \gr_{x}^{y}$ we have the following compatibility
  condition:
$$\int_{\gr_x}f(\gamma)d\mu_x(\gamma)
=\int_{\gr_y}f(\gamma \circ \eta)d\mu_y(\gamma)$$
\item for each $f\in \cg$ the map
$$x\mapsto \int_{\gr_x}f(\gamma)d\mu_x(\gamma) $$ belongs to $\cgo$
\end{itemize}

\end{definition}

A Lie groupoid always posses a smooth Haar system. In fact, if we
fix a smooth (positive) section of the 1-density bundle associated to
the Lie algebroid we obtain a smooth Haar system 
in a canonical way. 
We suppose for the rest of the
paper a given smooth Haar system given by 1-densities (for complete
details see \cite{Pat}). 
We can now define a convolution
product on $\cg$: Let $f,g\in \cg$, we set

$$(f*g)(\gamma)
=\int_{\gr_{s(\gamma)}}
f(\gamma \cdot \eta^{-1})g(\eta)d\mu_{s(\gamma)}(\eta)$$

This gives a well defined associative product. 
\begin{remark}
There
is a way to define the convolution algebra
using half densities (see Connes book \cite{Concg}).
\end{remark}

\subsection{Analytic indices for Lie groupoids}

As we mentioned in the introduction, we are going to consider some elements
in the $K$-theory group $K_0(\cg)$. We recall how these elements are
usually defined (See \cite{NWX} for complete details): First we recall a few facts about 
$\gr$-Pseudodifferential calculus:

A $\gr$-$\it{Pseudodifferential}$ $\it{operator}$ is a family of
pseudodifferential operators $\{ P_x\}_{x\in \go} $ acting in
$\ci_c(\gr_x)$ such that if $\gamma \in \gr $ and
$$U_{\gamma}:\ci_c(\gr_{s(\gamma)}) \rightarrow \ci_c(\gr_{r(\gamma)}) $$
the induced operator, then we have the following compatibility condition 
$$ P_{r(\gamma)} \circ U_{\gamma}= U_{\gamma} \circ P_{s(\gamma)}.$$
We also admit, as usual, operators acting in sections of vector bundles $E\rightarrow \go$.
There is also a differentiability condition with respect to $x$ that can
be found in \cite{NWX}.

In this work we are going to work exclusively with uniformly supported operators, let us recall this notion. Let $P=(P_x,x\in \go)$ be a $\gr$-operator, we denote by $k_x$ the Schwartz kernel pf $P_x$. Let 
$$supp\, P:=\overline{\cup_xsupp \,k_x}, \text{ and }$$
$$supp_{\mu}P:=\mu_1(supp\,P),$$ where $\mu_1(g',g)=g'g^{-1}$. 
We say that $P$ is uniformly supported if $supp_{\mu}P$ is compact.

We denote by $\Psi^m(\gr,E)$ the space of uniformly supported $\gr$-operators, acting on sections of a vector bundle $E$. We denote also
\begin{center}
$\Psi^{\infty}(\gr,E)=\bigcup_m \Psi^m(\gr,E)$ et $\Psi^{-\infty}(\gr,E)=\bigcap_m \Psi^m(\gr,E).$
\end{center}

The composition of two such operators is again of this kind (lemma 3, \cite{NWX}). 
In fact, $\Psi^{\infty}(\gr,E)$ is a filtered algebra (theorem 1, réf.cit.), $\it{i.e.}$, 
\begin{center}
$\Psi^{m}(\gr,E)\Psi^{m'}(\gr,E)\subset \Psi^{m+m'}(\gr,E).$
\end{center}
In particular,  $\Psi^{-\infty}(\gr,E)$ is a bilateral ideal.

\begin{remark}
The choice on the support justifies on the fact that  
$\Psi^{-\infty}(\gr,E)$ is identified with 
$\ci_c(\gr,End(E))$, thanks the Schwartz kernel theorem (theorem 6 \cite{NWX}). 
\end{remark}

The notion of principal symbol extends also to this setting. Let us denote by 
$\pi: A^*\gr \rightarrow \go$ the projection. For $P=(P_x,x\in \go)\in \Psi^{m}(\gr,E,F)$, 
the principal symbol of $P_x$, $\sigma_m(P_x)$, is a $\ci$ section of the vector bundle 
$End(\pi_x^*r^*E, \pi_x^*r^*F)$ over $T^*\gr_x$ (where 
$\pi_x:T^*\gr_x \rightarrow \gr_x$), such that at each fiber the morphism is homogeneous of degree $m$ 
(see \cite{AS} for more details). 
There is a section  
$\sigma_m(P)$ of $End(\pi^*E, \pi^*F)$ over $A^*\gr$ 
such that
\begin{equation}\label{gpsym}
\sigma_m(P)(\xi)=\sigma_m(P_x)(\xi)\in End(E_x,F_x) \text{ si } \xi \in A^*_x\gr
\end{equation}
Hence (\ref{gpsym}) above, induces a unique surjective linear map
\begin{equation}\label{gpsymap}
\sigma_m:\Psi^{m}(\gr,E)\rightarrow \sw^m(A^*\gr,End(E,F)),
\end{equation}
with kernel $\Psi^{m-1}(\gr,E)$ (see for instance 
proposition 2 \cite{NWX}) and where $\sw^m(A^*\gr,End(E,F))$ denotes the sections of the fiber 
$End(\pi^*E,\pi^*F)$ over $A^*\gr$ 
homogeneous of degree $m$ at each fiber.

\begin{definition}[$\gr$-Elliptic operators]
Let $P=(P_x,x\in \go)$ be a $\gr$-pseudodifferential operator. 
We will say that $P$ is elliptic if each $P_x$ is elliptic. 

We denote by $Ell(\gr)$ the set of $\gr$-pseudodifferential elliptic operators.
\end{definition}

The linear map  (\ref{gpsymap}) defines a principal symbol class 
$[\sigma(P)]\in K^0(A^*\gr)$:
\begin{equation}\label{ellgsymb}
Ell(\gr)\stackrel{\sigma}{\longrightarrow}K^0(A^*\gr).
\end{equation}

Connes, \cite{Coinc}, proved that if $P=(P_x,x\in \go)\in Ell(\gr)$, then it exists 
$Q\in  \Psi^{-m}(\gr,E)$ such that
\begin{center}
$I_E-PQ\in \Psi^{-\infty}(\gr,E)$ et $I_E-QP\in \Psi^{-\infty}(\gr,E)$,
\end{center}
where $I_E$ denotes the identity operator over $E$. In other words, $P$ defines a quasi-isomorphism in 
$(\Psi^{+\infty},\cg)$ and thus an element in $K_0(\cg)$ explicitly (when $E$ is trivial) given by  
\begin{equation}\label{index}
\left[T\left( 
\begin{array}{cc}
1 & 0\\
0 & 0
\end{array}
\right)
T^{-1}\right]- \left[ \left( 
\begin{array}{cc}
1 & 0\\
0 & 0
\end{array}
\right) \right] \in K_0(\widetilde{\cg}),
\end{equation}
where $1$ is the unit in $\widetilde{\cg}$ (unitarisation of $\cg$), and where $T$ is given by 
$$
T=\left( 
\begin{array}{cc}
(1-PQ)P+P & PQ-1\\
1-QP & Q
\end{array}
\right)
$$
with inverse
$$T^{-1}=\left( 
\begin{array}{cc}
Q & 1-QP\\
PQ-1 & (1-PQ)P+P
\end{array}
\right).$$

If $E$ is not trivial we obtain in the same way an element of 
$K_0(\ci_c(\gr,End(E,F)))\approx K_0(\cg)$ since $\ci_c(\gr,End(E,F)))$ is Morita
equivalent to $\cg$.

\begin{definition}[$\gr$-Index]
Let $P$ be a $\gr$-pseudodifferential elliptic operator. We denote by $ind\,P\in K_0(\cg)$ 
the element defined by $P$ as above. It is called the index of $P$. It defines a correspondence
\begin{equation}\label{ellgind}
Ell(\gr)\stackrel{ind}{\longrightarrow}K_0(\cg).
\end{equation}
\end{definition}

\begin{example}[The principal symbol class as a Groupoid index]
Let $\gr$ be a Lie groupoid. We can consider the Lie algebroid as Lie groupoid with its vector bundle structure $A\gr\rightrightarrows \go$. Let $P$ be a $\gr$-pseudodifferential elliptic operator, then 
the principal symbol $\sigma(P)$ is a $A\gr$-pseudodifferential elliptic operator. Its index, 
$ind(\sigma(P))\in K_0(\ci_c(A\gr))$ can be pushforward to $K_0(C_0(A^*\gr))$ using the inclusion of algebras 
$\ci_c(A\gr)\stackrel{j}{\hookrightarrow}C_0(A^*\gr)$ (modulo Fourier), the resulting image gives 
precisely the map (\ref{ellgsymb}) above, {\it i.e.}, $j_*(ind\,\sigma(P))=[\sigma(P)]\in K_0(C_0(A^*\gr))\approx 
K^0(A^*\gr)$.
\end{example}

We have a diagram
\[
\xymatrix{
Ell(\gr) \ar[r]^{ind}\ar[d]_{\sigma}& K_0(\cg) \\
K^0(A^*\gr) \ar@{.>}[ur] & ,
}
\]
where the pointed arrow does not always exist. It does in the classical cases, but not for general Lie groupoids as shown by the next example (\cite{Concg} pp. 142).

\begin{example}\label{exConn}
Let $\Rr \rightarrow \{0\}$ be the groupoid given by the group structure in $\Rr$. In 
\cite{Concg} (proposition 12, II.10.$\gamma$), Connes shows that the map 
$$D\mapsto ind\,D\in K_0(\ci_c(\Rr))$$ defines an injection of the projective space of no zero polynomials $D=P(\frac{\partial}{\partial x})$ into $K_0(\ci_c(\Rr))$.
\end{example}
We could consider the morphism 
\begin{equation}
K_0(\cg)\stackrel{j}{\longrightarrow} K_0(C^*_r(\gr))
\end{equation}
induced by the inclusion $\cg \subset C^*_r(\gr)$, then the composition 
$$Ell(\gr)\stackrel{ind}{\longrightarrow}K_0(\cg)\stackrel{j}{\longrightarrow} K_0(C^*_r(\gr))$$ does factors through the principal symbol class. In other words, we have the following commutative diagram
\[
\xymatrix{
Ell(\gr) \ar[r]^{ind}\ar[d]_{\sigma}& K_0(\cg)\ar[d]^j \\
K^0(A^*\gr) \ar[r]_{ind_a} & K_0(C^*_r(\gr)).
}
\]
Indeed,, $ind_a$ is the index morphism associated to the exact sequence of $C^*$-algebras (\cite{Coinc}, \cite{CS}, \cite{MP}, \cite{NWX})
\begin{equation}\label{gpdose}
0\rightarrow C^*_r(\gr)\longrightarrow \overline{\Psi^0(\gr)}\stackrel{\sigma}{\longrightarrow}C_0(S^*\gr)\rightarrow 0
\end{equation}
where $\overline{\Psi^0(\gr)}$ is a certain $C^*-$completion of $\Psi^0(\gr)$, $S^*\gr$ is the sphere vector bundle of $A^*\gr$ and $\sigma$ is the extension of the principal symbol.

\begin{definition}\label{gia}[$\gr$-Analytic index]
Let $\gr \rightrightarrows \go$ be a Lie groupoid. 
The morphism
\begin{equation}
K^0(A^*\gr) \stackrel{ind_a}{\longrightarrow}K_0(C^*_r(\gr))
\end{equation}
is called the analytic index of $\gr$.
\end{definition}

The $K$-theory of $C^*$-algebras has very good cohomological properties, however as we are going to discuss in the next subsection, it is sometimes preferable to work with the indices at the level of $\ci_c$-algebras.


\subsubsection{Pairing with Cyclic cohomology: Index formulas}

The interest to keep track on the $\ci_c$-indices is because at this level we can make explicit calculations via the Chern-Weil-Connes theory. In fact there is a pairing 
\cite{Concdg,Concg,Karhc}
\begin{equation}\label{accouplement}
\langle \_ \, , \_ \rangle :K_0(\cg)\times HP^*(\cg)\rightarrow \mathbb{C}
\end{equation}
There are several known cocycles over $\cg$. An important problem in Noncommutative Geometry is to compute the above pairing in order to obtain numerical invariants from the indices in  $K_0(\cg)$, \cite{Concg,CMnov,GorLottfg}. Let us illustrate this affirmation with the following example.

\begin{example}\label{intronov}\cite{CMnov,Concg}
Let $\Gamma$ be a discrete group acting properly and freely on a smooth manifold $\tilde{M}$ with compact quotient
$\tilde{M}/\Gamma:=M$. Let $\gr\rightrightarrows \go=M$ be the Lie groupoid quotient of $\tilde{M}\times
\tilde{M}$ by the diagonal action of $\Gamma$. 

Let $c \in H^{*}(\Gamma):=H^*(B\Gamma)$. Connes-Moscovici showed in \cite{CMnov} that the higher Novikov signature, $Sign_c(M)$, 
can be obtained with the pairing of the signature operator $D_{sign}$
and a cyclic cocycle $\tau_c$ associated to $c$: 
\begin{equation}\label{signind}
\langle \tau_c , ind\, D_{sgn}\rangle=Sign_c(M,\psi). 
\end{equation}
The Novikov conjecture states that these higher signatures are oriented homotopy invariants of
$M$. Hence, if
$ind\, D_{sign}\in K_0(\cg)$ is a homotopy invariant of
$(M,\psi)$ then the Novikov conjecture would follow. We only know that 
$j(ind\, D_{sign})\in K_0(C^*_r(\gr))$ is a homotopy invariant. But then we have to extend the action of $\tau_c$ to $K_0(C^*_r(\gr))$. Connes-Moscovici show that this action extends for Hyperbolic groups.
\end{example}

The pairing (\ref{accouplement}) above is not interesting for $C^*$-algebras. Indeed, the Cyclic cohomology for $C^*$-algebras is trivial (see \cite{CTS} 5.2 for an explanation). In fact, as shown by the example above, a very interesting problem is to compute the pairing at the $\ci_c$-level and then extend the action of the cyclic cocycles to the $K$-theory of the $C^*$-algebra. This problem is known as the extension problem and it was solved by Connes for some cyclic cocycles associated to foliations, \cite{Cotfc}, and by Connes-Moscovici, \cite{CMnov}, for group cocycles when the group is hyperbolic. 

The most general formula for the pairing (\ref{accouplement}), known until these days (as far the author is aware), is the one of Gorokhovsky-Lott for Foliation groupoids (\cite{GorLottfg}, theorem 5.) which generalized a previous Connes formula for {\'e}tale groupoids (\cite{Concg}, theorem 12, III.7.$\gamma$, see also \cite{GorLotteg} for a superconnection proof). It basically says the following: Let $\gr\rightrightarrows M$ be a foliation groupoid (Morita equivalent to an {\'e}tale groupoid). It carries a foliation $\F$. Let $\rho$ be a closed holonomy-invariant transverse current on $M$. Suppose $\gr$ acts freely, properly and cocompactly on a manifold $P$. Let $D$ be a $\gr$-elliptic differential operator on $P$. Then the following formula holds:
\begin{equation}
\langle Ind\, D, \rho\rangle=\int_{P/\gr}\hat{A}(TF)ch([\sigma_D])\nu^*(\omega_{\rho}),
\end{equation}
where $\omega_{\rho}\in H^*(B\gr,o)$ is a universal class associated to $\rho$ and $\nu:P/\gr\rightarrow B\gr$ is a classifying map.

Now, we can expect an easy (topological) calculation only if the map 
$D\mapsto \langle D \, , \tau \rangle$ 
($\tau \in HP^*(\cg)$ fix) factors through the symbol class of $D$, $[\sigma(D)]\in K^0(A^*\gr)$: we want 
to have a diagram of the following kind:
\[
\xymatrix{
Ell(\gr) \ar[r]^-{ind} \ar[d]_{\sigma} & K_0(\cg) \ar[rr]^-{\langle \_ , \tau \rangle} & & \mathbb{C} \\
K^0(A^*\gr) \ar@{.>}[urrr]_-{\tau} & &  &.
}
\]

This paper is concerned with the solution of the factorization problem justed described. Our approach will use a geometrical deformation associated to any Lie groupoid, known as the tangent groupoid. We will discuss this in the next section.

\section{Index theory and strict deformation quantization}

\subsection{Deformation to the normal cone}

The tangent groupoid is a particular case of a geometric construction that we describe here.

Let $M$ be a $\ci$ manifold and $X\subset M$ be a $\ci$ submanifold. We denote
by $\Nb_{X}^{M}$ the normal bundle to $X$ in $M$, $\it{i.e.}$, 
$\Nb_{X}^{M}:= T_XM/TX$.

We define the following set
\begin{align}
\Dnc_{X}^{M}:= \Nb_{X}^{M} \times {0} \bigsqcup M \times \Rr^* 
\end{align} 
The purpose of this section is to recall how to define a $\ci$-structure in $\Dnc_{X}^{M}$. This is more or less classical, for example
it was extensively used in \cite{HS}.

Let us first consider the case where $M=\Rr^p\times \Rr^q$ 
and $X=\Rr^p \times \{ 0\}$ (where we
identify canonically $X=\Rr^p$). We denote by
$q=n-p$ and by $\Dnc_{p}^{n}$ for $\Dnc_{\Rr^p}^{\Rr^n}$ as above. In this case
we clearly have that $\Dnc_{p}^{n}=\Rr^p \times \Rr^q \times \Rr$ (as a
set). Consider the 
bijection  $\psi: \Rr^p \times \Rr^q \times \Rr \rightarrow
\Dnc_{p}^{n}$ given by 
\begin{equation}\label{psi}
\psi(x,\xi ,t) = \left\{ 
\begin{array}{cc}
(x,\xi ,0) &\mbox{ if } t=0 \\
(x,t\xi ,t) &\mbox{ if } t\neq0
\end{array}\right.
\end{equation}
which inverse is given explicitly by 
$$
\psi^{-1}(x,\xi ,t) = \left\{ 
\begin{array}{cc}
(x,\xi ,0) &\mbox{ if } t=0 \\
(x,\frac{1}{t}\xi ,t) &\mbox{ if } t\neq0
\end{array}\right.
$$
We can consider the $\ci$-structure on $\Dnc_{p}^{n}$
induced by this bijection.

We pass now to the general case. A local chart 
$(\Uo,\phi)$ in $M$ is said to be a $X$-slice if 
\begin{itemize}
\item[1)]$\phi : \Uo \stackrel{\cong}{\rightarrow} U \subset \Rr^p\times \Rr^q$
\item[2)]If $\Uo \cap X =\Vo$, $\Vo=\phi^{-1}( U \cap \Rr^p \times \{ 0\}
  )$ (we denote $V=U \cap \Rr^p \times \{ 0\}$)
\end{itemize}
With this notation, $\Dnc_{V}^{U}\subset \Dnc_{p}^{n}$ as an
open subset. We may define a function 
\begin{equation}\label{phi}
\tilde{\phi}:\Dnc_{\Vo}^{\Uo} \rightarrow \Dnc_{V}^{U} 
\end{equation}
in the following way: For $x\in \Vo$ we have $\phi (x)\in \Rr^p
\times \{0\}$. If we write 
$\phi(x)=(\phi_1(x),0)$, then 
$$ \phi_1 :\Vo \rightarrow V \subset \Rr^p$$ 
is a diffeomorphism. We set 
$\tilde{\phi}(v,\xi ,0)= (\phi_1 (v),d_N\phi_v (\xi ),0)$ and 
$\tilde{\phi}(u,t)= (\phi (u),t)$ 
for $t\neq 0$. Here 
$d_N\phi_v: N_v \rightarrow \Rr^q$ is the normal component of the
 derivative $d\phi_v$ for $v\in \Vo$. It is clear that $\tilde{\phi}$ is
 also a  bijection (in particular it induces a $C^{\infty}$ structure on $\Dnc_{\Vo}^{\Uo}$). 
Now, let us consider an atlas 
$ \{ (\Uo_{\alpha},\phi_{\alpha}) \}_{\alpha \in \Delta}$ of $M$
 consisting of $X-$slices. Then the collection $ \{ (\Dnc_{\Vo_{\alpha}}^{\Uo_{\alpha}},\tilde{\phi_{\alpha})}
  \} _{\alpha \in \Delta }$ is a $\ci$-atlas of
  $\Dnc_{X}^{M}$ (proposition 3.1 in \cite{Ca2}).

\begin{definition}[Deformation to the normal cone]
Let $X\subset M$ be as above. The set
$\Dnc_{X}^{M}$ equipped with the  $C^{\infty}$ structure
induced by the atlas described in the last proposition is called
$\it{"The\, deformation\, to\, normal\, cone\, associated\, to\,}$   
$X\subset M$". 
\end{definition}

\begin{remark}
Following the same steps, we can define a deformation to the normal
cone associated to an injective immersion $X\hookrightarrow M$.
\end{remark}

One important feature about this construction is that it is in
some sense functorial. More explicitly, let $(M,X)$ 
and $(M',X')$ be $\ci$-pairs as above and let
 $F:(M,X)\rightarrow (M',X')$
be a pair morphism, i.e., a $\ci$ map   
$F:M\rightarrow M'$, with $F(X)\subset X'$. We define 
$ \Dnc(F): \Dnc_{X}^{M} \rightarrow \Dnc_{X'}^{M'} $ by the following formulas:\\

$\Dnc(F) (x,\xi ,0)= (F(x),d_NF_x (\xi),0)$ and\\

$\Dnc(F) (m ,t)= (F(m),t)$ for $t\neq 0$,
\noindent
where $d_NF_x$ is by definition the map
\[  (\Nb_{X}^{M})_x 
\stackrel{d_NF_x}{\longrightarrow}  (\Nb_{X'}^{M'})_{F(x)} \]
induced by $ T_xM 
\stackrel{dF_x}{\longrightarrow}  T_{F(x)}M'$.

Then $\Dnc(F):\Dnc_{X}^{M} \rightarrow \Dnc_{X'}^{M'}$ is a $\ci$-map (proposition 3.4 in \cite{Ca2}).

\subsection{The tangent groupoid}


\begin{definition}[Tangent groupoid]
Let $\gr \rightrightarrows \go $ be a Lie groupoid. $\it{The\, tangent\,
groupoid}$ associated to $\gr$ is the groupoid that has $\Dnc_{\go}^{\gr}$ as the set of arrows and  $\go \times \Rr$ as the units, with:
\begin{itemize}
\item[$\cdot$] $s^T(x,\eta ,0) =(x,0)$ and $r^T(x,\eta ,0) =(x,0)$ at $t=0$.
\item[$\cdot$] $s^T(\gamma,t) =(s(\gamma),t)$ and $r^T(\gamma,t)
  =(r(\gamma),t)$ at $t\neq0$.
\item[$\cdot$] The product is given by
  $m^T((x,\eta,0),(x,\xi,0))=(x,\eta +\xi ,0)$ et \linebreak $m^T((\gamma,t), 
  (\beta ,t))= (m(\gamma,\beta) , t)$ if $t\neq 0 $ and 
if $r(\beta)=s(\gamma)$.
\item[$\cdot$] The unit map $u^T:\go \rightarrow \gr^T$ is given by
 $u^T(x,0)=(x,0)$ and $u^T(x,t)=(u(x),t)$ for $t\neq 0$.
\end{itemize}
We denote $\gr^{T}:= \Dnc_{\go}^{\gr}$ and $A\gr:=\Nb_{\go}^{\gr}$.
\end{definition} 

As we have seen above $\gr^{T}$ can be considered as a $\ci$ manifold with
boundary. As a consequence of the functoriality of the Deformation to the normal cone,
one can show that the tangent groupoid is in fact a Lie
groupoid. Indeed, it is easy to check that if we identify in a
canonical way $\Dnc_{\go}^{\gr^{(2)}}$ with $(\gr^T)^{(2)}$, then 
$$ m^T=\Dnc(m),\, s^T=\Dnc(s), \,  r^T=\Dnc(r),\,  u^T=\Dnc(u)$$
where we are considering the following pair morphisms:
\begin{align}  
m:((\gr)^{(2)},\go)\rightarrow (\gr,\go ), \nonumber
\\
s,r:(\gr ,\go) \rightarrow (\go,\go),\nonumber 
\\
u:(\go,\go)\rightarrow (\gr,\go ).\nonumber
\end{align}

\begin{remark}\label{haartg}
Finally, let $\{ \mu_x\}$ be a smooth Haar system on $\gr$, {\it i.e.}, a choice of $\gr$-invariant Lebesgue measures. In particular we have an associated smooth Haar system on $A\gr$ (groupoid given by the vector bundle structure), which we denote again by $\{ \mu_x\}$. Then the following family $\{\mu_{(x,t)}\}$ is a smooth Haar system for the tangent groupoid of $\gr$ (details may be found in \cite{Pat}):
\begin{itemize}
\item $\mu_{(x,0)}:=\mu_x$ at $(\gr^T)_{(x,0)}=A_x\gr$ and 
\item $\mu_{(x,t)}:=t^{-q}\cdot \mu_x$  at $(\gr^T)_{(x,t)}=\gr_x$ for
  $t\neq 0$, where $q=dim\, \gr_x$.
\end{itemize}
In this article, we are only going to consider these Haar systems for the tangent groupoids.
\end{remark}

\subsubsection{Analytic indices for Lie groupoids as deformations}

Let $\gr \rightrightarrows \go$ be a Lie groupoid and
$$K^0(A^*\gr)\stackrel{ind_a}{\longrightarrow} K_0(C_{r}^{*}(\gr)),$$
its analytic index. This morphism can also be constructed using the tangent groupoid and its $C^*$-algebra.

It is easy to check that the evaluation morphisms extend to the $C^*$-algebras:
 $$C_r^*(\gr^T) \stackrel{ev_0}{\longrightarrow}
C_r^*(A\gr) \text{ and }$$ $$C_r^*(\gr^T) \stackrel{ev_t}{\longrightarrow}
C_r^*(\gr) \text{ for } t\neq 0.$$
Moreover, since $\gr \times (0,1]$ is an open saturated subset of $\gt$ and $A\gr$ an open saturated closed subset, we have the following exact sequence (\cite{HS})
\begin{equation}\label{segt}
0 \rightarrow C_r^*(\gr \times (0,1]) \longrightarrow C_r^*(\gr^T)
\stackrel{ev_0}{\longrightarrow}
C_r^*(A\gr) \rightarrow 0.
\end{equation}
Now, the $C^*$-algebra $C_r^*(\gr \times (0,1])\cong C_0((0,1],C_r^*(\gr))$ is contractible. 
This implies that the groups $K_i(C_r^*(\gr \times (0,1]))$ vanish, for $i=0,1$. 
Then, applying the $K-$theory  functor to the exact sequence above, we obtain that
$$K_i(C_r^*(\gr^T)) \stackrel{(ev_0)_*}{\longrightarrow}
K_i(C_r^*(A\gr))$$ is an isomorphism, for $i=0,1$. In \cite{MP},
Monthubert-Pierrot show that
\begin{equation}\label{defind}
ind_a=(ev_1)_*\circ
(ev_0)_*^{-1},
\end{equation}
modulo the Fourier isomorphism identifying $C_r^*(A\gr)\cong
C_0(A^*\gr)$ (see also \cite{HS} and \cite{NWX}).
Putting this in a commutative diagram, we have
\begin{equation}\label{eldiagramaevaluadoestrella}
\xymatrix{
&K_0(C_r^*(\gr^T))\ar[ld]_{e_0}^{\approx} \ar[rd]^{e_1}&\\
K^0(A^*\gr)\ar[rr]_{ind_{a}}&&K_{0}(C^*_r(\gr)). 
}
\end{equation}
Compare the last diagram with (\ref{eldiagramaevaluado}) above.

The algebra $C_r^*(\gr^T)$ is a strict deformation quantization of $C_0(A^*\gr)$, and the analytic index morphism of $\gr$ can be constructed by means of this deformation. In the next section we are going to discuss the existence of a strict deformation quantization algebra associated the tangent groupoid but in more primitive level, that is, not a $C^*$-algebra but a Schwartz type algebra. We will use afterwards to define other index morphisms as deformations.

\subsection{A Schwartz algebra for the tangent groupoid}


In this section we will recall how to construct the deformation algebra mentioned at the introduction.
For complete details, we refer the reader to \cite{Ca2}.

The Schwartz algebra for the Tangent groupoid will be a particular
case of a construction associated to any deformation to the normal
cone. 

\begin{definition}\label{ladef}
Let $p, q\in \Nat$ and $U \subset \Rr^p
  \times \Rr^q$ 
an open subset, and let $V=U\cap (\Rr^p \times \{ 0\})$.
\begin{itemize}
\item[(1)]Let $K\subset U \times \Rr$ be a compact
  subset. We say that $K$ is a conic compact subset of $U \times \Rr$
relative to $V$ if
\[ K_0=K\cap (U \times \{ 0\}) \subset V\]

\item[(2)]Let $\Omega_{V}^{U}=\{(x,\xi,t)\in 
\Rr^p \times \Rr^q \times \Rr: (x,t\cdot \xi)\in U \},$ 
which is an open subset of $\Rr^p \times \Rr^q \times \Rr$ and thus
a $\ci$ manifold. 
Let $g \in \ci (\Omega_{V}^{U})$. We say that
   $g$ has compact conic support, if there exists a conic
  compact $K$
 of $U \times \Rr$ relative to $V$ such that if 
$(x, t\xi ,t) \notin K$ then $g(x, \xi ,t)=0$.

\item[(3)]We denote by $\src (\Omega_{V}^{U})$ 
the set of functions
$g\in \ci (\Omega_{V}^{U})$ 
that have compact conic support and that satisfy the following condition:

\begin{itemize}
\item[$(s_1$)]$\forall$ $k,m\in \Nat$, $l\in \Nat^p$
and $\alpha \in \Nat^q$ there exist $C_{(k,m,l,\alpha)} >0$ such that
\[ (1+\| \xi \|^2)^k \| \partial_{x}^{l}\partial_{\xi}^{\alpha}
\partial_{t}^{m}g(x,\xi ,t) \| \leq C_{(k,m,l,\alpha)}   \]
\end{itemize}

\end{itemize}

\end{definition}

Now, the spaces $\src (\Omega_{V}^{U})$ are invariant under
diffeomorphisms. More precisely: Let $F:U\rightarrow U'$ be a $\ci$-diffeomorphism such that $F(V)=V'$; let 
$\tilde{F}:\Omega_{V}^{U}\rightarrow \Omega_{V'}^{U'}$ be the induced map. Then, for every 
$g\in \src (\Omega_{V'}^{U'})$, we have that
$\tilde{g}:= g\circ \tilde{F} \in \src (\Omega_{V}^{U})$ (proposition 4.2 in \cite{Ca2}).

This compatibility result allows to give the following definition.

\begin{definition}\label{src}
Let $g \in \ci (\Dnc_{X}^{M}) $.
\begin{itemize}

\item[(a)]We say that $g$ has conic compact support $K$, if there exists a compact subset
 $K\subset M \times \Rr$ with $K_0:=K\cap (M\times \{ 0\}) \subset X$ (conic
 compact relative to $X$) such that if $t\neq 0$ and 
$(m,t) \notin K$ then $g(m,t)=0$.

\item[(b)]We say that $g$ is rapidly decaying at zero if for every
$(\Uo,\phi)$  $X$-slice chart
and for every $\chi \in \ci_c(\Uo \times \Rr)$, the map 
$g_{\chi}\in \ci(\Omega_{V}^{U})$ ($\Omega_{V}^{U}$ as in definition \ref{ladef}.)
given by
\[ g_{\chi}(x,\xi ,t)= (g\circ \varphi^{-1})(x,\xi ,t) 
\cdot (\chi \circ p \circ \varphi^{-1})(x,\xi ,t) \]
is in  $\src (\Omega_{V}^{U})$, where 
\begin{itemize}
\item[$\cdot$] $p$ is the deformation of the pair map $(M,X)\stackrel{Id}{\longrightarrow}
(M,M)$, {\it i.e.}, $p:\Dnc_{X}^{M} \rightarrow M\times \Rr$ is given by $(x,\xi,0)\mapsto (x,0)$, and 
$(m,t)\mapsto (m,t)$ for $t\neq 0$, and
\item[$\cdot$] $\varphi:=\tilde{\phi}^{-1}\circ \psi :\Omega_{V}^{U}\rightarrow \Dnc_{\Vo}^{\Uo}$, where $\psi$ and $\tilde{\phi}$ are defined at (\ref{psi}) and (\ref{phi}) above.
\end{itemize}
\end{itemize}

Finally, we denote by $\src (\Dnc_{X}^{M})$ the set of functions 
$g\in \ci(\Dnc_{X}^{M})$ that are rapidly decaying at zero 
with conic compact support.
\end{definition}

\begin{remark}\label{remsrc}
\noindent
\begin{itemize}
\item[{\it (a)}] Obviously $\ci_c(\Dnc_{X}^{M})$ is a subspace of $\src (\Dnc_{X}^{M})$.
\item[{\it (b)}] Let 
$\{ (\Uo_{\alpha},\phi_{\alpha}) \}_{\alpha \in \Delta}$ be a family of
$X-$slices covering $X$. We have a decomposition of $\src(\Dnc_{X}^{M})$ as follows (see remark 4.5 in \cite{Ca2} and discussion below it):
\begin{align}\label{decomposicion}
\src (\Dnc_{X}^{M}) = \sum_{\alpha \in \Lambda} \src (\Dnc_{\Vo_{\alpha}}^{\Uo_{\alpha}}) 
+ \ci_c (M\times \Rr^*).
\end{align}
\end{itemize}
\end{remark}

The main theorem in \cite{Ca2} (Theorem 4.10) is the following

\begin{theorem}
The space $\src (\gr^T)$ is stable under convolution, and we have the following inclusions of algebras
$$\ct \subset \src (\gr^T) \subset C_r^*(\gr^T)$$ 
Moreover, $\src (\gr^T)$ is a field of algebras over $\Rr$, whose fibers are
$$\sw (A\gr) \text{ at } t=0, \text{ and }$$
$$\cg \text{ for } t\neq 0.$$  
\end{theorem}  

In the statement of this theorem, $\sw (A\gr)$ denotes the Schwartz algebra over the Lie algebroid. Let us briefly recall the notion of Schwartz space associated to a
vector bundle: For a trivial bundle $X\times \Rr^q\rightarrow X$, $\sw(X\times \Rr^q):=\ci_c(X,\sw(\Rr^q))$ (see \cite{trev}). In general, $\sw(E)$ is defined using local charts. More precisely, a partition of the unity argument, allows to see that if we take a covering of $X$, $\{(\Vo_{\alpha},\tau_{\alpha})\}_{\alpha \in \Delta}$, consisting on trivializing charts, then we have a decomposition of the following kind:
\begin{equation}\label{des}
\sw(E)= \sum_{\alpha}\sw (\Vo_{\alpha}\times \Rr^q).
\end{equation}

The "Schwartz algebras" have in general the good $K-$theory
groups. As we said in the introduction, we are
interested in the group $K^0(A^*\gr)=K_0(C_0(A^*\gr))$. It is not 
enough to take the $K-$theory of $\ci_c(A\gr)$ (see
example \ref{exConn}). As we showed in \cite{Ca4} (proposition 4.5), $\sw (A^*\gr)$ has the wanted $K$-theory, {\it i.e.}, 
$K^0(A^*\gr)\cong K_0(\sw(A\gr))$. In particular, our deformation algebra restricts at zero to the right algebra. 

From now on it will be important to restrict our functions on the tangent groupoid to the closed interval 
$[0,1]$. We keep the notation $\src (\Dnc_{X}^{M})$ for the restricted space. All the results above remain true. So for instance $\src (\gr^T)$ is an algebra which is a field of algebras over the closed interval $[0,1]$ with 0-fiber $\sw(A\gr)$ and $\ci_c(\gr)$ otherwise.

We have the following short exact sequence of algebras (\cite{Ca4}, proposition 4.6):
\begin{align}\label{se}
0 \longrightarrow J \longrightarrow \src (\gr^T)
\stackrel{e_0}{\longrightarrow} \sw (A\gr) \longrightarrow 0,
\end{align}
where $J=Ker(e_0)$ by definition.

\section{Higher localized indices}
\begin{definition}
Let $\tau$ be a (periodic) cyclic cocycle over $\cg$. We say that $\tau$ can be localized if 
the correspondence
\begin{equation}
\xymatrix{
Ell(\gr)\ar[r]^-{ind}&K_0(\cg)\ar[r]^-{\langle\, \_, \tau \rangle}&\mathbb{C}
}
\end{equation}
factors through the principal symbol class morphism. In other words, if there is a unique morphism $K^0(A^*\gr)\stackrel{Ind_{\tau}}{\longrightarrow}\mathbb{C}$ which fits in the following commutative diagram
\begin{equation}
\xymatrix{
Ell(\gr)\ar[r]^-{ind} \ar[d]_{[psymb]}&K_0(\cg)\ar[r]^-{\langle\, \_, \tau \rangle}&\mathbb{C}\\
K^0(A^*\gr)\ar[rru]_-{Ind_{\tau}}&&
}
\end{equation}
{\it i.e.}, satisfying $Ind_{\tau}(a)=\langle ind\, D_a,\tau \rangle$, and hence completely characterized by this property. 
In this case, we call $Ind_{\tau}$ the higher localized index associated to $\tau$.
\end{definition}

\begin{remark}
If a cyclic cocycle can be localized then the higher localized index $Ind_{\tau}$ is completely characterized by the property: $Ind_{\tau}([\sigma_D])=\langle ind\, D,\tau \rangle$, $\forall D\in Ell(\gr)$.
\end{remark}

We are going to prove first a localization result for Bounded cyclic cocycles, we recall its definition.

\begin{definition}\label{defBCC}
A multilinear map $\tau:\underbrace{\cg \times \cdots \times \cg}_{q+1-times}
\rightarrow \mathbb{C}$ is bounded if it extends to a continuous multilinear map 
$\underbrace{\ck \times \cdots \times \ck}_{q+1-times}\stackrel{\tau_k}{\longrightarrow}\mathbb{C}$, for some $k\in \mathbb{N}$. 
\end{definition}

We can re-state theorem 6.9 in \cite{Ca4} in the following way:

\begin{theorem}\label{teotesis}
Let $\gr\rightrightarrows \go$ be a Lie groupoid, then\\
\begin{itemize}
\item[{\it (i)}]Every bounded cyclic cocycle over $\cg$ can be localized. \\

\item[{\it (ii)}]Moreover, if the groupoid is {\'e}tale, then every cyclic cocycle can be localized.
\end{itemize}
\end{theorem}

We will recall the main steps for proving this result. For this purpose we need to define the intermediate group 
$$K_0(\cg)\rightarrow K_0^B(\gr)\rightarrow K_0(C_r^*(\gr)).$$
Let us denote, for each $k\in \mathbb{N}$, $K_{0}^{h,k}(\gr)$ the quotient group of $K_0(\ck)$ by the equivalence relation induced by $\tiny{K_0(\ckt)\overset{e_0}{\underset{e_1}{\rightrightarrows}}K_0(\ck)}$. Let $K_{0}^{F}(\gr)= \varprojlim_{k}K_{0}^{h,k}(\gr)$ be the projective limit relative to the inclusions $\ck \subset
C_{c}^{k-1}(\gr)$. 
We can take the inductive limit $$\varinjlim_{m} K_{0}^{F}(\gr \times \Rr^{2m})$$ 
induced by $K_{0}^{F}(\gr \times \Rr^{2m})\stackrel{Bott}{\longrightarrow}
K_{0}^{F}(\gr \times \Rr^{2(m+1)})$ (the Bott morphism). We denote this group by 
\begin{equation}\label{KBF}
K_{0}^{B}(\gr):=\varinjlim_{m} K_{0}^{F}(\gr \times \Rr^{2m}),
\end{equation}

Now, theorem 5.4 in \cite{Ca4} establish the following two assertions:
\begin{enumerate}

\item There is a unique group morphism
$$ ind_{a}^{B}: K^0(A^*\gr)\rightarrow K_{0}^{B}(\gr) $$
that fits in the following commutative diagram
\begin{equation}\label{eldiagramaevaluado}
\xymatrix{
&K_0(\src(\gr^T))\ar[ld]_{e_0} \ar[rd]^{e_1^B}&\\
K^0(A^*\gr)\ar[rr]_{ind_{a}^{B}}&&K_{0}^{B}(\gr), 
}
\end{equation}
where $e_1^B$ is the evaluation at one $K_0(\src(\gt))\stackrel{e_1}{\longrightarrow}K_0(\cg)$ followed by the canonical map $K_0(\cg)\rightarrow K_0^B(\gr)$.

\item This morphism also fits in the following commutative diagram
\begin{equation}\label{eldiagrama}
\xymatrix{
Ell(\gr) \ar[d]_{\sigma} \ar[r]^{ind}&
K_0(\ci_c(\gr)) \ar[d] & \\ K^0(A^*\gr) \ar[d]_{id} \ar[r]^{ind_{a}^{B}}& K_{0}^{B}(\gr) \ar[d] & \\
K^0(A^*\gr) \ar[r]_{ind_a}& K_0(C_{r}^{*}(\gr)) & .
}
\end{equation}
\end{enumerate}

Next, it is very easy to check (see \cite{Ca4} Proposition 6.7) that if $\tau$ is a bounded cyclic cocycle, then the pairing morphism 
$K_0(\cg)\stackrel{\langle \, ,\tau \rangle}{\longrightarrow}$ extends to $K_0^B(\gr)$, {\it i.e.}, we have a commutative diagram of the following type:
\begin{equation}\label{pairF}
\xymatrix{
K_0(\cg) \ar[r]^-{<,\tau>} \ar[d]_-{\iota} & \mathbf{C} \\
K_0^B(\gr) \ar[ru]_-{\tau_B} &
}
\end{equation}

Now, theorem \ref{teotesis} follows immediately because we can put together diagrams (\ref{eldiagrama}) and (\ref{pairF}) to get the following commutative diagram
\begin{equation}\label{eldiagramacompleto}
\xymatrix{
Ell(\gr) \ar[r]^-{ind} \ar[d]_-{\sigma} & K_0(\cg) \ar[d] \ar[rr]^-{\langle \_ , \tau \rangle} & & \mathbb{C} \\
K^0(A^*\gr) \ar[r]_-{ind_{a}^{B}}& K_{0}^{B}(\gr) \ar[urr]_-{\tau_B} &  &.
}
\end{equation}

\subsection{Higher localized index formula}

In this section we will give a formula for the Higher localized indices in terms of a pairing in the strict deformation quantization algebra $\src(\gt)$. We have first to introduce some notation :

Let $\tau$ be a $(q+1)-$multilinear functional over $\cg$. For each $t\neq 0$, we let 
$\tau_t$ be the $(q+1)$-multilinear functional over $\src (\gt)$ defined by
\begin{equation}
\tau_t(f^0,...,f^q):=\tau(f^0_t,...,f^q_t).
\end{equation}

In fact, if we consider the evaluation morphisms 
$$e_t:\src(\gt)\rightarrow \cg,$$
for $t\neq0$, then it is obvious that $\tau_t$ is a $(b,B)$-cocycle (periodic cyclic cocycle) over $\src(\gt)$ if $\tau$ is a $(b,B)$-cocycle over $\cg$. Indeed, $\tau_t=e_t^*(\tau)$ by definition.

We can now state the main theorem of this article.

\begin{theorem}\label{teoindloc}
Let $\tau$ be a bounded cyclic cocycle then the higher localized index of $\tau$, $K^0(A^*\gr)\stackrel{Ind_{\tau}}{\longrightarrow}\mathbb{C}$, 
is given by
\begin{equation}\label{tauindloc}
Ind_{\tau}(a)=lim_{t\rightarrow 0}\langle \widetilde{a},\tau_t \rangle,
\end{equation}
where $\widetilde{a}\in K_0(\src(\gt))$ is such that  $e_0(\widetilde{a})=a\in K^0(A^*\gr)$.
In fact the pairing above is constant for $t\neq 0$.
\end{theorem}

\begin{remark}
Hence, if $\tau$ is a bounded cyclic cocycle and $D$ is a $\gr$-pseudodifferential elliptic operator, then we have the following formula for the pairing: 
\begin{equation}
\langle ind\, D,\tau \rangle=\langle \widetilde{\sigma_D},\tau_t \rangle,
\end{equation}
for each $t\neq 0$, and where $\widetilde{\sigma_D}\in K_0(\src(\gt))$ is such that $e_0(\widetilde{\sigma})=\sigma_D$. In particular, 
\begin{equation}
\langle ind\, D,\tau \rangle=lim_{t\rightarrow 0}\langle \widetilde{\sigma_D},\tau_t \rangle.
\end{equation}
\end{remark}

For the proof of the theorem above we will need the following lemma.

\begin{lemma}\label{lemat}
For $s,t \in (0,1]$, $\tau_s$ and $\tau_t$ define the same pairing map 
$$K_0(\src(\gt))\longrightarrow \mathbb{C}.$$
\end{lemma}

\begin{proof}
Let $p$ be an idempotent in $\widetilde{\src(\gt)}=\src(\gt)\oplus\mathbb{C}$. It defines a smooth family of idempotents $p_t$ in $\widetilde{\src(\gt)}$. 
We set $a_t:=\frac{dp_t}{dt}(2p_t-1)$. 
Hence, a simple calculation shows 
$$\frac{d}{dt} \langle \tau ,p_t \rangle= \sum_{i=0}^{2n}\tau(p_t,...,[a_t,p_t],...,p_t)=:L_{a_t}\tau(p_t,...,p_t).$$
Now, {\it{the Lie derivatives}} $L_{x_t}$ act trivially on $HP^0(\src(\gt))$ (see \cite{Concdg,Good}), then $ \langle \tau ,p_t \rangle$ is constant in $t$. Finally, by definition, 
$\langle \tau_t ,p \rangle=\langle \tau ,p_t \rangle$. Hence $t\mapsto \langle \tau_t ,p \rangle$ is a constant function for $t\in (0,1]$.
\end{proof}

\begin{proof}[Proof of theorem \ref{teoindloc}]
Putting together diagrams (\ref{eldiagramaevaluado}) and (\ref{eldiagramacompleto}), we get the following commutative diagram
\begin{equation}\label{eldiagramalocal}
\xymatrix{
Ell(\gr) \ar[r]^-{ind} \ar[d]_-{\sigma} & K_0(\cg) \ar[d] \ar[rrd]^-{\langle \_ , \tau \rangle} & &  \\
K^0(A^*\gr) \ar[r]^-{ind_{a}^{B}}& K_{0}^{B}(\gr) \ar[rr]_-{\tau_B} &  &\mathbb{C}\\
&K_0(\src(\gt))\ar[u]_{e_1^B}\ar[rru]_{\tau_1}\ar[lu]^{e_0}&&.
}
\end{equation}
In other words, for $a\in K^0(A^*\gr)$, 
$Ind_{\tau}(a)=\langle \widetilde{a},\tau_1 \rangle.$
Now, by lemma \ref{lemat} we can conclude that
$$Ind_{\tau}(a)=\langle \widetilde{a},\tau_t \rangle,$$
for each $t\neq 0$. In particular the limit when $t$ tends to zero is given by 
$$Ind_{\tau}(a)=lim_{t\rightarrow 0}\langle \widetilde{a},\tau_t \rangle.$$
\end{proof}

For {\'e}tale groupoids, we can state the following corollary.

\begin{corollary}\label{corindloc}
If $\gr \rightrightarrows \go$ is an {\'e}tale groupoid, then formula (\ref{tauindloc}) holds for every cyclic cocycle.
\end{corollary}

\begin{proof}
Thanks to the works of Burghelea, Brylinski-Nistor and Crainic (\cite{Bur,BN,Cra}), we know a very explicit description of the Periodic cyclic cohomology for {\'e}tale groupoids. For instance, we have a decomposition of the following kind (see for example \cite{Cra} theorems 4.1.2. and 4.2.5)
\begin{equation}\label{Crainic}
HP^*(\cg)=\Pi_{\ops}H_{\tau}^{*+r}(B\Nb_{\ops}),
\end{equation}
where $\Nb_{\ops}$ is an {\'e}tale groupoid associated to $\ops$ (the normalizer of 
$\ops$, see 3.4.8 in ref.cit.). For instance, when $\ops=\go$, $\Nb_{\ops}=\gr$.

Now, all the cyclic cocycles coming from the cohomology of the classifying space are bounded. Indeed, we know that each factor of 
$HP^*(\cg)$ in the decomposition (\ref{Crainic}) consists of bounded cyclic cocycles (see last section of \cite{Ca4}). 
Now, the pairing 
$$HP^*(\cg)\times K_0(\cg)\longrightarrow \mathbb{C}$$
is well defined. In particular, the restriction to $HP^*(\cg)|_{\ops}$ vanishes for almost every $\ops$. The conclusion is now immediate from the theorem above.
\end{proof}

Once we have the formula (\ref{tauindloc}) above, it is well worth it to recall why the evaluation morphism 
\begin{equation}\label{evaluationzero}
K_0(\src(\gt))\stackrel{e_0}{\longrightarrow}K^0(A^*\gr)
\end{equation}
is surjective. Let $[\sigma] \in K_0(\sw (A\gr))=K^0(A^*\gr)$. 
We know from the $\gr$-pseudodifferential calculus that $[\sigma]$ can be represented by a smooth homogeneous elliptic symbol (see \cite{AS,Connes_Higson,MP,NWX}). We can consider the symbol over $A^*\gr
\times [0,1]$ that coincides with $\sigma$ for all $t$, we denote it 
by $\tilde{\sigma}$. Now, since
$A\gr^T=A\gr
\times [0,1] $, we can take $\tilde {P}=(P_t)_{t\in [0,1]}$ a 
$\gr^T$-elliptic pseudodifferential operator associated to $\sigma$, that is,
$\sigma_{\tilde{P}}=\tilde{\sigma}$. Let
$i:\ci_c (\gr^T) \rightarrow \src (\gr^T)$ be the inclusion (which is an algebra morphism), then 
$i_*(ind\, \tilde{P}) \in K_0(\src
(\gr^T))$ is such that 
$e_{0,*}(i_*(ind\, \tilde{P}))=[\sigma]$. Hence, the lifting of a principal symbol class is given by the index of $\tilde {P}=(P_t)_{t\in [0,1]}$ and theorem \ref{teoindloc} says that the pairing with a bounded cyclic cocycle does not depend on the choice of the operator $P$. Now, for compute this index, as in formula (\ref{index}), one should find first a parametrix for the family $\tilde {P}=(P_t)_{t\in [0,1]}$.

For instance, in \cite{CMnov} (section 2), Connes-Moscovici consider elliptic differential operators over compact manifolds, let us say an operator $D\in DO^r(M;E,F)^{-1}$. Then they consider the family of operators $tD$ (multiplication by $t$ in the normal direction) for $t>0$ and they construct a family of parametrix $\widetilde{Q}(t)$. The corresponding idempotent is then homotopic to  $W(tD)$, where
\begin{equation}
W(D)=\left( 
\begin{array}{cc}
e^{-D^*D} & e^{-\frac{1}{2}D^*D}(\frac{I-e^{-D^*D}}{D^*D})^{\frac{1}{2}}D^*\\
e^{-\frac{1}{2}DD^*}(\frac{I-e^{-DD^*}}{DD^*})^{\frac{1}{2}}D & I-e^{-D^*D}
\end{array}
\right)
\end{equation}
is the Wasserman idempotent. In the language of the tangent groupoid, the family $\tilde{D}=\{D_t\}_{t\in [0,1]}$ where $D_0=\sigma_D$ and $D_t=tD$ for $t>0$, defines a $\gt$-differential elliptic operator. What Connes and Moscovici compute is precisely the limit on right hand side of formula (\ref{tauindloc}).

Also, in \cite{MW} (section 2), Moscovici-Wu proceed in a similar way by using the finite propagation speed property to construct a parametrix for operators $\tilde{D}=\{D_t\}_{t\in [0,1]}$ over the tangent groupoid. Then they obtain as associated idempotent the so called graph projector. What they compute after is again a particular case of the right hand side of (\ref{tauindloc}).

Finally, in \cite{GorLotteg} (section 5.1), Gorokhovsky-Lott use the same technics as the two previous examples in order to obtain their index formula. 

\begin{remark}
As a final remark is interesting to mention that in the formula (\ref{tauindloc}) both sides make always sense. In fact the pairing 
$$\langle \widetilde{a},\tau_t \rangle$$
is constant for $t\neq 0$. 

We could then consider the differences 
$$\eta_{\tau}(D):=\langle ind\, D,\tau \rangle-lim_{t\rightarrow 0}\langle \widetilde{\sigma_D},\tau_t \rangle$$
for any $D$ pseudodifferential elliptic $\gr$-operator.

For Lie groupoids (the base is a smooth manifold) these differences do not seem to be very interesting, however it would be interesting to adapt the methods and results of this paper to other kind of groupoids or higher structures, for example to Continuous families groupoids, groupoids associated to manifolds with boundary or with conical singularities (\cite{Mont,DLN}). Then probably these kind of differences could give interesting data. See \cite{MW,LP05} for related discussions.
\end{remark}


\bibliographystyle{amsalpha}
\bibliography{bibliographie}

\end{document}